\documentclass[a4paper,11pt]{amsart}

\usepackage[latin1]{inputenc} 
\usepackage{graphics}
\usepackage[matrix,arrow,tips,curve]{xy}
\usepackage[english]{babel}
\usepackage{amsmath}
\usepackage{amssymb}

\usepackage{tikz-cd}

\usepackage{mathrsfs}

\usepackage{enumerate}

\newtheorem{thm}{Theorem}
\newtheorem{lemma}[thm]{Lemma}
\newtheorem{corollary}[thm]{Corollary}
\newtheorem{proposition}[thm]{Proposition}

\newtheorem*{thm*}{Theorem}

\theoremstyle{definition}

\newtheorem{example}[thm]{Example}
\newtheorem{remark}[thm]{Remark}
\newtheorem{construction}[thm]{Construction}

\newcommand{\ph}{\varphi}

\newcommand{\la}{\longrightarrow}
\newcommand{\ol}{\mathcal{O}}

\newcommand{\pr}{\mathbb{P}}
\newcommand{\Q}{\mathbb{Q}}

\newcommand{\R}{\mathbb{R}}
\newcommand{\Z}{\mathbb{Z}}
\newcommand{\N}{\mathcal{N}_1}
\newcommand{\Nu}{\mathcal{N}^1}

\newcommand{\Sing}{\operatorname{Sing}}
\newcommand{\Pic}{\operatorname{Pic}}

\newcommand{\NE}{\operatorname{NE}}
\newcommand{\Exc}{\operatorname{Exc}}

\newcommand{\Lo}{\operatorname{Locus}}
\newcommand{\codim}{\operatorname{codim}}

\newcommand{\Bl}{\operatorname{Bl}}

\usepackage{etoolbox}
\patchcmd{\section}{\normalfont}{\normalfont\large}{}{}
\patchcmd{\subsection}{\bfseries}{\scshape\centering}{}{}
\patchcmd{\subsection}{-.5em}{.5em}{}{}

\title{The Lefschetz defect of Fano varieties}
\author{C.~Casagrande}
\address{Universit\`a di Torino,
Dipartimento di Matematica,
via Carlo Alberto 10,
10123 Torino - Italy}
\email{cinzia.casagrande@unito.it}
\date{November 30, 2022}

\begin{document}

\maketitle

\noindent
Smooth, complex Fano varieties are a very natural class of projective varieties, and have been intensively studied both classically and in recent times, due to their significance in the framework of the Minimal Model Program.

The Lefschetz defect is an invariant of smooth Fano varieties that has been recently introduced in \cite{codim}; it is related to the Picard number $\rho_X$ of $X$, and to the Picard number of prime divisors in $X$.
This paper is a survey on this new invariant and its properties: we explain the definition of the Lefschetz defect $\delta_X$ and its origin,
the known results, and several examples, especially in dimensions $3$ and $4$;  we do not give new theoretical results.

In particular, we explain how the results on the Lefschetz defect allow to recover the classification of Fano $3$-folds with $\rho_X\geq 5$, due to Mori and Mukai. We also review  the known examples of Fano $4$-folds with $\rho_X\geq 6$, that are remarkably few:
apart products and toric examples, there are only 6 known families, with $\rho_X\in\{6,7,8,9\}$.







\tableofcontents

\section{Glossary}
\noindent We collect here the notation and terminology used in the paper; we suggest the reader to skip this section, and come back when needed.

Let $X$ be a projective variety.\\
$\N(X)$ (respectively $\Nu(X)$) is the vector space of one-cycles  (respectively Cartier divisors) in $X$, with coefficients in $\R$, up to numerical equivalence.\\
$[\,\text{-}\,]$ stands for the numerical equivalence class of a curve or a divisor.\\
$\rho_X=\dim\N(X)=\dim\Nu(X)$ is the Picard number of $X$.\\
$\NE(X)$ is the cone of effective curves in $\N(X)$, an extremal ray $R$ of $\NE(X)$ is a one-dimensional face.
If $D$ is a divisor, $D\cdot R>0$ means $D\cdot \gamma>0$ for any non-zero class $\gamma\in R$.

A contraction is a surjective morphism with connected fibers $\ph\colon X\to Y$ between normal projective varieties; $\ph$ is elementary if $\rho_X-\rho_Y=1$.

A conic bundle is a contraction of fiber type such that every fiber is isomorphic to a plane conic.

A small modification is a birational map $X\dasharrow X'$ that is an isomorphism in codimension one.

When $X$ is smooth and Fano, or more generally of Fano type,\footnote{A smooth projective variety $X$ is of Fano type if there exists an effective $\Q$-divisor $\Delta$ such that the pair $(X,\Delta)$ is log Fano, namely $(X,\Delta)$ has klt singularities and $-(K_X+\Delta)$ is ample, see \cite{prokshok}.} then $\NE(X)$ is closed and polyhedral, and there is a bijection between extremal rays $R$ of $\NE(X)$ and elementary contractions $\ph\colon X\to Y$, given by $R=\NE(X)\cap\ker \ph_*$, where $\ph_*\colon\N(X)\to\N(Y)$ is the pushforward. This is because varieties of Fano type are Mori dream spaces by \cite[Corollary 1.3.2]{BCHM}. We also set $\Lo(R):=\Exc(\ph)$.

For any closed subset $Z\subset X$ we set $\N(Z,X):=\iota_*(\N(Z))$, where  $\iota\colon Z \hookrightarrow X$ is the inclusion and $\iota_*\colon\N(Z)\to\N(X)$ is the induced pushforward.

$\Bl_m\pr^2$ is the blow-up of $\pr^2$ in $m$ points in general linear position.

A $\pr^r$-bundle is the projectivization of a  vector bundle.

We work over the field of complex numbers.
\section{The definition of Lefschetz defect and the case $\delta_X\geq 4$}\label{defi}
\noindent Let $X$ be a smooth  Fano variety. We introduce an invariant of $X$, the Lefschetz defect, which relates the Picard number of $X$ to that of its prime divisors.

Given a prime divisor $D$ in $X$, let us consider the natural pushforward $\iota_*\colon\N(D)\to\N(X)$ given by the inclusion $\iota\colon D\hookrightarrow X$, and set
$$\N(D,X):=\iota_*(\N(D))\subseteq\N(X).$$
This is the linear span, in $\N(X)$, of classes of curves contained in $D$. The Lefschetz defect of $X$ is defined as:
$$\delta_X:=\max\bigl\{\codim\N(D,X)\,|\,D\subset X\text{ a prime divisor}\bigr\}.$$
\noindent{\bf Remarks}\label{remarks}
\begin{enumerate}[{\bf 1.}]
\item\label{range}
  The linear subspace $\N(D,X)$ is always non-zero,
because if $C\subset D$ is any curve, we have $[C]\neq 0$ in $\N(X)$. Therefore
 $\delta_X\in\{0,\dotsc,\rho_X-1\}$; in particular we see that $\delta_X=0$ whenever $\rho_X=1$.
\item\label{surfaces}
When $\dim X=2$, a prime divisor $D\subset X$ is an irreducible curve, and $\N(D,X)=\R[D]$, thus $\delta_X=\rho_X-1$.
\item We can consider the dual picture:
  let $D\subset X$ be a prime divisor, and consider the restriction $r_D\colon H^2(X,\R)\to H^2(D,\R)$; then we have $\codim\N(D,X)=\dim\ker r_D$.
  
  Indeed there is a commutative diagram:
  $$
\begin{tikzcd}
\Nu(X) \arrow[r, "\tilde{r}_D"] \arrow[d, "\wr"]&{\Nu(D)}\arrow[d, hook]\\
    {H^2(X,\R)}\arrow[r, "r_D"]& {H^2(D,\R)}
\end{tikzcd}
$$
where $\tilde{r}_D$ is dual to $\iota_*$, and $\Nu(X)\cong H^2(X,\R)$ because $X$ is Fano.
\item
 One can also define the Lefschetz defect as:
 $$\delta_X=\max\bigl\{\dim\ker\bigl(r_D\colon H^2(X,\R)\to H^2(D,\R)\bigr)\,|\,D\text{ an effective divisor in }X\bigr\}$$
 (where with a slight abuse of notation we identify $D$ with its support).

 Indeed let $D$ be an effective divisor, and $D_1$ an irreducible component of $D$. Then $r_{D_1}$ factors as $H^2(X,\R)\to H^2(D,\R)\to H^2(D_1,\R)$, so that $\ker r_D\subseteq\ker r_{D_1}$, and the maximum above is always attained for some prime divisor.

 If $\dim X\geq 3$ and $D$ is ample, then $r_D$ is injective by Lefschetz's theorem on hyperplane sections, so that the Lefschetz defect depends on non-ample effective divisors, and measures the failure of this instance of Lefschetz's theorem.
\end{enumerate}
\setcounter{thm}{4}
 \begin{lemma}[the Lefschetz defect of a product variety] 
   If $X\cong S\times T$, then $\delta_X=\max\{\delta_S,\delta_T\}$.

   In particular, if $S$ is a surface, then $\delta_X=\max\{\rho_S-1,\delta_T\}$.
 \end{lemma}
 \begin{proof}
   We can assume
  that $\delta_S\geq\delta_T$. Let $D_S\subset S$ be a prime divisor with $\codim\N(D_S,S)=\delta_S$. Then $\codim\N(D_S\times T,X)=\delta_S$, so that $\delta_X\geq \delta_S$.

Conversely,
  let $D\subset X$ be a prime divisor, and $\pi\colon X\to S$ the projection.

  If $\pi(D)\subsetneq S$, then $D=\pi(D)\times T$ and $\codim \N(D,X)=\codim\N(\pi(D),S)\leq \delta_S$.

  Suppose instead that $\pi(D)=S$, and consider the pushforward $\pi_*\colon\N(X)\to\N(S)$. Then we have 
$\pi_*(\N(D,X))=\N(S)$, and hence $\dim\N(D,X)=\rho_S+\dim(\N(D,X)\cap\ker\pi_*)$.

  Let $p\in S$; we have a natural identification $\N(T)\cong \N(\pi^{-1}(p), X)=\ker\pi_*$. Moreover $\pi^{-1}(p)\cap D$ is either $T$ or a divisor in $T$; in both cases $\dim(\N(D,X)\cap\ker\pi_*)\geq \dim\N(\pi^{-1}(p)\cap D,T)\geq \rho_T-\delta_T\geq\rho_T-\delta_S$, so that
  $\dim\N(D,X)\geq \rho_X-\delta_S$, and we conclude that $\delta_X= \delta_S$.

  Finally if $\dim S=2$ we have $\delta_S=\rho_S-1$ by Remark \ref{surfaces}.
\end{proof} 
   \begin{example}[$\delta_X$ and disjoint divisors]\label{disjoint}
Let $D,E_1,\dotsc,E_r\subset X$ be prime divisors, and assume that $D\cap(E_1\cup\cdots\cup E_r)=\emptyset$. If the classes $[E_1],\dotsc,[E_r]\in\Nu(X)$ are linearly independent, then $\delta_X\geq\codim\N(D,X)\geq r$.

Indeed for every curve $C\subset D$ we have $E_i\cdot C=0$ for every $i=1,\dotsc,r$, so that $[C]\in E_1^{\perp}\cap\cdots\cap E_r^{\perp}$ (where $E_i^{\perp}\subset\N(X)$ is the hyperplane of classes orthogonal to $E_i$), and finally $\N(D,X)\subseteq E_1^{\perp}\cap\cdots\cap E_r^{\perp}$.

For instance, as soon as $X$ contains two disjoint prime divisors, we have $\delta_X>0$.
\end{example}
The next example concerns the case $\dim X=3$, and shows how birational geometry can be used to find prime divisors $D$ with small $\dim\N(D,X)$.
\begin{example}[\cite{gloria}, Lemma 5.1]\label{es3fold}
  Let $X$ be a Fano $3$-fold.
  We show that $X$ always contains a prime divisor $D$ with $\dim\N(D,X)\leq 2$, so that $\delta_X\in\{\rho_X-2,\rho_X-1\}$.

  Let $\ph\colon X\to Y$ be an elementary contraction. It is well-known that $\ph$ cannot be small: it is either divisorial, or of fiber type. In any case there exists a prime divisor $D\subset X$ such that $\dim \ph(D)\leq 1$.

  Consider the pushforward $\ph_*\colon \N(X)\to \N(Y)$; this is a surjective linear map with one-dimensional kernel. We have $\ph_*(\N(D,X))=\N(\ph(D),Y)$; on the other hand $\N(\ph(D),Y)=\{0\}$ if $\ph(D)=\{pt\}$, and  $\N(\ph(D),Y)=\R[\ph(D)]$ if $\ph(D)$ is a curve. In any case we have $\dim\N(\ph(D),Y)\leq 1$ and hence $\dim\N(D,X)\leq 2$.
  \end{example}
The following is the main property of the Lefschetz defect.
  \begin{thm}[\cite{codim}, Th.~3.3]\label{codim}
    For every smooth Fano variety $X$ we have $\delta_X\in\{0,\dotsc,8\}$.

    Moreover, if $\delta_X\geq 4$, then $X\cong S\times T$ where $S$ is a del Pezzo surface with
$\rho_S=\delta_X+1$, and $\delta_T\leq\delta_X$.
\end{thm}
Thus, if a Fano variety $X$ is not a product, then $\delta_X\in\{0,1,2,3\}$. 
\begin{remark}[bounding the Picard number of $X$]
Let $X$ be a smooth Fano variety and  $D\subset X$ a prime divisor.
We have $\dim\N(D,X)\leq\rho_D$ and hence 
  $\delta_X\geq\rho_X-\rho_D$.  Thus Theorem \ref{codim} implies that $\rho_X\leq\rho_D+8$, and if  $\rho_X\geq\rho_D+4$, then $X$ is a product.

Note however that this result does not immediately give   an inductive bound on $\rho_X$ in terms of the dimension of $X$, because
we do not know whether $X$ contains a prime divisor $D$ which is smooth and Fano.

It is indeed expected that $\rho_X\leq\frac{9}{2}\dim X$; this is attained in every even dimension $2m$ by products $X=S_1\times\cdots\times S_m$ of degree one del Pezzo surfaces, which have $\rho_X=9m$ and $\delta_X=8$. We will discuss the maximal Picard number for Fano $3$-folds and $4$-folds in Corollary \ref{3folds} and \S \ref{sec4folds} respectively.
\end{remark}
\begin{remark}[the origin of the notion of Lefschetz defect]\label{origin}
  Bonavero, Campana, and Wi{\'s}niewski   \cite{bonwisncamp} have classified smooth Fano varieties obtained by blowing-up a point in a smooth variety; this has been later generalized by Tsukioka \cite{toru} to smooth Fano varieties containing a divisor $D\cong\pr^{n-1}$ with negative normal bundle. In both papers, the main strategy is to consider the divisor $D\cong\pr^{n-1}$, an extremal ray $R$ of $\NE(X)$ such that $D\cdot R>0$, and the associated elementary contraction $\ph\colon X\to Y$.
  
The pushforward $\ph_*\colon \N(X)\to \N(Y)$ is a surjective linear map with one-dimensional kernel.
  Since $D\cong\pr^{n-1}$, the space $\N(D,X)$ is one-dimensional, and we have two possibilities: either $\ker\ph_*=\N(D,X)$  and $\ph(D)=\{pt\}$, or $\ker\ph_*\cap\N(D,X)=\{0\}$ and $\ph$ is finite on $D$. In this last case, if $F\subset X$ is a non-trivial fiber of $\ph$, we have $D\cap F\neq \emptyset$ (because $D$ has positive intersection with every curve contracted by $\ph$) and $\dim(D\cap F)=0$, so that $\dim F=1$.
  On the other hand,
  elementary contractions of smooth Fano varieties having only fibers of dimension $\leq 1$ are classified \cite{ando,wisn}: they are either a conic bundle, or a blow-up of a smooth, codimension $2$ subvariety in a smooth variety.

  In this way one gets a very special elementary contraction of $X$, and this argument may be iterated in the target $Y$ of $\ph$. Here the relevant property is that $\dim\N(D,X)=1$, more than $D\cong\pr^{n-1}$.

  Later, in \cite{31}, the author has studied Fano varieties having a birational elementary contraction $f\colon X\to X'$ sending the exceptional divisor $E$ to a curve (for instance the blow-up of a smooth curve in $X'$). It is not difficult to see that $\dim\N(E,X)=2$ in this case (see Example \ref{es3fold}), and
  prime divisors $D$ with $\dim\N(D,X)=2$ play a special role in   \cite{31}. This assumption is still applied by studying
  extremal rays $R$ such that $D\cdot R>0$, and the associated elementary contractions.

  Finally this led the author to generalize this strategy in
  \cite{codim}, using birational geometry and the MMP. We give an overview in the following section.
\end{remark}  
\section{Studying $\codim\N(D,X)$ via birational geometry}\label{MMP}
\noindent Let $X$ be a smooth Fano variety and $D\subset X$ a prime divisor.
In this section we explain a construction, based on birational geometry, which allows to get information on  $X$ as soon as $\codim\N(D,X)>0$.
We refer the reader to \cite{kollarmori,hukeel} for the terminology in birational geometry and Mori dream spaces, and to \cite[\S 2]{codim} for more details.

Since Fano varieties are Mori dream spaces 
\cite[Corollary 1.3.2]{BCHM}, one can run a MMP for any divisor  \cite[Proposition 2.2]{codim}.  We consider here
a MMP for $-D$, which means that we consider \emph{extremal rays having positive intersection with $D$}, as in Remark \ref{origin}. Since $-D$  can never become semiample in the MMP, the program must end with a fiber type contraction. This means that we get a sequence
$$
X=X_0\stackrel{\sigma_0}{\dasharrow} X_1 \dasharrow\quad\cdots
\quad
\dasharrow X_{k-1}\stackrel{\sigma_{k-1}}{\dasharrow} X_k\stackrel{\ph}{\la}Y
$$
such that:
\begin{enumerate}[$\bullet$]
\item every $X_i$ is a normal and $\Q$-factorial projective variety, 
and the transform  $D_i\subset X_i$  of $D$ is a prime divisor in $X_i$;
\item for every $i=0,\dotsc,k-1$ there exists an extremal
  ray $R_i$ of $X_i$ such that $D_i\cdot
  R_i>0$,  $\Lo(R_i)\subsetneq X_i$, and 
  $\sigma_i$ is either the contraction of $R_i$ (if $R_i$ is
  divisorial), or its flip (if $R_i$ is small);
\item  
there exists an extremal ray $R_k$ in $X_k$, with a fiber type
  contraction $\ph\colon X_k\to Y$,
such
  that $D_k\cdot
  R_k>0$.
\end{enumerate}

Moreover, since $X$ is Fano, using an MMP with scaling of $-K_X$, in the sequence above we can also assume that $-K_{X_i}\cdot R_i>0$ for every $i=0,\dotsc,k$ \cite[Proposition 2.4]{codim}.

Then we analyse how $\codim\N(D_i,X_i)$ varies along the steps of the MMP. It is not difficult to show that, for 
 every $i=0,\dotsc,k-1$, we have:
$$\codim\N(D_{i+1},X_{i+1})=\begin{cases} \codim\N(D_i,X_i)\quad&\text{if $R_i\subset\N(D_i,X_i)$,}\\
\codim\N(D_i,X_i)-1\quad&\text{if $R_i\not\subset\N(D_i,X_i)$}
\end{cases}
$$
(see \cite[Lemma 2.6(3)]{codim}).

Let us consider now the last step, given by the fiber type contraction $\ph\colon X_k\to Y$. Since $D_k\cdot R_k>0$, the prime divisor $D_k\subset X_k$ intersects every fiber of $\ph$, namely $\ph(D_k)=Y$. 
Let $\ph_*\colon\N(X_k)\to\N(Y)$
be the pushforward; then 
$\ph_*(\N(D_k,X_k))=\N(\ph(D_k),Y)=\N(Y)$. On the other hand $\ph$ is elementary, thus $\dim\ker\ph_*=1$, and we get:
$$\codim\N(D_k,X_k)\leq 1.$$

In conclusion: \emph{we have at least $\codim\N(D,X)-1$ steps where $\codim\N(D_i,X_i)$ drops, namely where $R_i\not\subset\N(D_i,X_i)$.} These steps are very special, let us describe them.

Fix $i\in\{0,\dotsc,k-1\}$ such that $R_i\not\subset\N(D_i,X_i)$, and let $\psi\colon X_i\to Z$ be the associated elementary birational contraction (divisorial or small). 
  
Let $F\subset X_i$ be a non-trivial fiber of $\psi$. We have
$D_i\cdot R_i>0$, thus $D_i$ has positive intersection with every curve contracted by $\psi$, so that
$D_i\cap F\neq \emptyset$. On the other hand,  a curve $C\subset D_i\cap F$ would have class $[C]\in\N(D_i,X_i)\cap R_i$, which is impossible because $R_i\not\subset\N(D_i,X_i)$ by assumption. We conclude that $\dim (D_i\cap F)=0$ and
 $\dim F=1$.

It turns out that since $-K_{X_i}\cdot R_i>0$,  the property of having fibers of dimension $\leq 1$ makes $\psi$ very special. If $X_i$ is smooth, then
\cite{ando,wisn} yield that $\psi$ is divisorial (so that  $\psi=\sigma_i\colon X_i\to X_{i+1}$)
and is the blow-up of a smooth, codimension $2$ subvariety.  

In general $X_i$ is singular, but a careful analysis shows that $\psi=\sigma_i\colon X_i\to X_{i+1}$ is a divisorial contraction, that $\Exc(\sigma_i)$ is contained in the open subset where the birational map $X\dasharrow X_i$ is an isomorphism (in particular, it is contained in the smooth locus of $X_i$), and that $\sigma_i$ is the blow-up of  a smooth, codimension $2$ subvariety contained in the smooth locus of $X_{i+1}$ \cite[Lemma 2.7(1)]{codim}. For this argument, it is crucial
to have the assumption that $-K_{X_j}\cdot R_j>0$ for every $j$.

Let $E_i\subset X$ be the transform of $\Exc(\sigma_i)\subset X_i$. Then $E_i$ is a smooth prime divisor, with a $\pr^1$-bundle structure, such that if $f_i\subset E_i$ is a fiber, we have $E_i\cdot f_i=-1$, $D\cdot f_i>0$, and $[f_i]\not\in \N(D,X)$ \cite[Lemma 2.7(3)]{codim}.

 In the end we get the following.
\begin{proposition}[\cite{codim}, Proposition 2.5]\label{MP2}
Let $X$ be a smooth Fano variety and $D\subset X$ a prime
divisor with $\codim\N(D,X)>0$. 

Then there exist
pairwise disjoint smooth prime divisors $E_1,\dotsc,E_s\subset X$, with
$s=\codim\N(D,X)-1$,
such that
 every $E_j$ is a $\pr^1$-bundle with  $E_j\cdot
f_j=-1$, where  $f_j\subset E_j$ is a fiber;
moreover $D\cdot f_j>0$ and $[f_j]\not\in\N(D,X)$. In particular
$E_j\cap D\neq\emptyset$ and $E_j\neq D$.
\end{proposition}
This construction is the starting point of the proof of Theorem \ref{codim}. The rest of the proof \cite[\S 3]{codim} is very technical and we will not enter into it.

\medskip

Let us give now 
an application to Fano varieties with
maximal Lefschetz defect.
We have seen in Remark \ref{range} that $\delta_X\leq \rho_X-1$;
the condition $\delta_X=\rho_X-1$ is equivalent to asking that $X$ contains a prime divisor $D$ with $\dim\N(D,X)=1$, for instance a prime divisor $D$ with $\rho_D=1$. This implies that $\rho_X\leq 9$ by Theorem \ref{codim}, but when $\dim X\geq 3$, the stronger bound $\rho_X\leq 3$ holds. This was originally proved in \cite{toru}; we give a different proof using the previous proposition.
\begin{proposition}
  \label{rho-1}
Let $X$ be a smooth Fano variety with $\dim X\geq 3$ and $\delta_X=\rho_X-1$. Then $\rho_X\leq 3$. 
\end{proposition}
\begin{proof}
  By contradiction, suppose that $\rho_X\geq 4$, and let $D\subset X$ be a prime divisor with $\dim\N(D,X)=1$. By Proposition \ref{MP2}, there exist $s=\rho_X-2\geq 2$ pairwise disjoint prime divisors $E_1,\dotsc,E_s$, all intersecting $D$, and distinct from $D$.
  
Since $\dim X\geq 3$, we have $\dim (D\cap E_1)\geq 1$; let $C\subset D\cap E_1$ be an irreducible curve.
 We have $E_2\cdot C=0$ because $E_1\cap E_2=\emptyset$. On the other hand $\dim\N(D,X)=1$, hence for every irreducible curve $C'\subset D$, there exists some $\lambda\in\Q$ such that $[C']=\lambda[C]$, thus $E_2\cdot C'=0$. Finally $E_2$ meets $D$, $E_2\neq D$, and $E_2$ has intersection zero with every curve in $D$, which is impossible.
\end{proof}  
The boundary case where $\rho_X=3$ and $\delta_X=2$ is described by a structure theorem \cite[Th.~3.8]{minimal}. In particular, Fano $4$-folds with
$\rho_X=3$ and $\delta_X=2$ are classified and studied in detail in
\cite{saverio}; they form 28 families, and only 3 of them are toric.

\medskip

We conclude this section by considering the case of Fano $3$-folds.
Example \ref{es3fold} and Proposition \ref{rho-1}  immediately imply the following.
\begin{lemma}\label{dim3}
  Let $X$ be a Fano 3-fold. Then $\delta_X\in\{\rho_X-2,\rho_X-1\}$, and $\delta_X=\rho_X-2$ as soon as $\rho_X\geq 4$.
\end{lemma}
\begin{table}[!h]\caption{The Lefschetz defect of Fano $3$-folds}\label{table3folds}
$$\begin{array}{||c|c||}
\hline\hline
   
    \rho=1   & \delta=0 \\
    
    \hline

   \rho=2   & 
              \delta\in\{0,1\}\\

  \hline

    \rho=3 &  \delta\in\{1,2\} \\

    \hline

    \rho\geq 4  & \delta=\rho-2\\

\hline\hline
  \end{array}$$
\end{table}

\begin{corollary}\label{3folds}
  Let $X$ be a Fano 3-fold with $\rho_X\geq 6$. Then $X\cong S\times\pr^1$ where $S$ is a del Pezzo surface. In particular $\rho_X\leq 10$.
\end{corollary}
This corollary was originally  proved by Mori and Mukai \cite[Theorem 1.2]{morimukai2}, via an explicit study of conic bundle structures on Fano 3-folds, and it was the conclusion of their classification of Fano $3$-folds with $\rho_X\geq 2$. 
\begin{proof}
  Since $\rho_X\geq 6$, by Lemma \ref{dim3} we have $\delta_X=\rho_X-2\geq 4$, and Theorem \ref{codim} yields the statement.
\end{proof}
In the next section  (Corollary \ref{class}) we will also recover the classification of Fano $3$-folds with $\rho_X=5$.
\section{Fano varieties with Lefschetz defect $3$}\label{secdelta3}
\noindent The bound $\delta_X\geq 4$ in Theorem \ref{codim} is sharp, as in every dimension $\geq 3$ there are Fano varieties with Lefschetz defect $3$ that are not products. On the other hand, the case $\delta_X=3$ is still very special: it was first partially studied in \cite{codim,delta3_4folds}, and then completely classified in \cite{delta3}. Let us now describe this classification, which  gives an explicit geometrical construction for Fano varieties with $\delta_X=3$.
\begin{construction}\label{construction}
 We start with   a smooth Fano variety $T$ with $\delta_T\leq 3$,
 and $L_1,L_2,L_3\in\Pic(T)$ such that:
 \begin{equation}
   \label{condition}
   -K_T+L_i-L_j\text{ is ample on $T$ for every }i,j=1,2,3.
   \end{equation}
We consider the $\pr^2$-bundle:
$$Z:=\pr_T(L_1\oplus L_2\oplus L_3)\stackrel{\ph}{\la}T,$$
and let $S_2,S_3\subset Z$ be the sections of $\ph$ corresponding to the projections $\oplus_{j=1}^3L_j\twoheadrightarrow L_2$ and 
$\oplus_{j=1}^3L_j\twoheadrightarrow L_3$ respectively.

The construction has two variants, that we call $A$ and $B$, as follows:
\begin{enumerate}
  \item[case $A$:]
let $S_1\subset Z$ be the section of $\ph$ corresponding to  the projection $\oplus_{j=1}^3L_j\twoheadrightarrow L_1$.
\item[case $B$:] we set $L_2=L_3=\ol_T$, and assume that $L_1\not\cong\ol_T$.
Let $H$ be the tautological line bundle of $Z$. We assume that the complete intersection $S_1$ of a general element in the linear system $|H|$ and one in $|2H|$ is smooth.\footnote{For instance, if $L_1$ is globally generated, then the vector bundle
$L_1\oplus\ol_T\oplus\ol_T$ is globally generated, thus $H$ is too, and by Bertini $S_1$ is smooth.}
\end{enumerate}
\begin{lemma}[\cite{delta3}, Rem.~4.2 and 4.3]\label{effective}
  The subvarieties $S_1,S_2,S_3$ are smooth, irreducible, of codimension 2, and pairwise disjoint.

  In case B, $\ph_{|S_1}\colon S_1\to T$ is finite of degree $2$, and $h^0(T,2L_1)>0$.
\end{lemma}
Finally let $X\to Z$ be the blow-up of $S_1,S_2,S_3$, so that $X$ is a smooth projective variety with $\dim X=\dim T+2$ and $\rho_X=\rho_T+4$.
\end{construction}
\begin{thm}[\cite{delta3}, Prop.~1.2 and 1.3, Th.~1.4]\label{delta3}
The variety $X$ is Fano with $\delta_X=3$. Moreover every smooth Fano variety with Lefschetz defect $3$ is obtained in this way.
\end{thm}
It is easy to exhibit a prime divisor in $X$ which realizes $\delta_X=3$, by looking at the exceptional divisors of the blow-up $X\to Z$.
Recall that $S_2$ is a section of $\ph$, so that $S_2\cong T$. Let $E\subset X$ be the exceptional divisor over $S_2$; then $E$ is a $\pr^1$-bundle over $S_2$ and $\rho_E=1+\rho_T=\rho_X-3$. One can check that the natural map $\N(E)\to\N(X)$ is injective, hence $\codim\N(E,X)=3$.

\medskip

Given $T$, it is not difficult to classify the possible choices of $L_i$'s as in cases $A$ and $B$ above; as an example, we describe the case $T=\pr^1$.
\begin{example}[the $3$-dimensional case]\label{ex}
  Let $X$ be a Fano $3$-fold. We assume that $\delta_X=3$,  equivalently that $\rho_X=5$ (see Table \ref{table3folds} on page \pageref{table3folds}).

 By  Theorem \ref{delta3}, $X$ is obtained with Construction \ref{construction} from a Fano variety $T$ of dimension $\dim X-2=1$, so that $T=\pr^1$. 

 We set $L_i:=\ol_{\pr^1}(a_i)$ for $i=1,2,3$.

  In case $A$, the $L_i$'s are determined up to shift and permutation, so that we can assume $a_3=0$ and $a_1\geq a_2\geq 0$. Condition \eqref{condition} yields $a_i\leq 1$ for $i=1,2$, so that  for $(a_1,a_2)$ we have the possibilities  $(0,0),(1,0),(1,1)$.

  On the other hand, in case $A$ we also have a commutative diagram:
   $$\xymatrix{ X\ar[r]\ar[d]& 
  {\pr_T((-L_1)\oplus(-L_2)\oplus(-L_3))}\ar[d]\\
  Z=\pr_T(L_1\oplus L_2\oplus L_3)\ar[r]&T}$$
(see \cite[Lemma 6.3]{delta3}), which fiberwise over $T$ is just the standard Cremona transformation in $\pr^2$. This means that the choice $L_1=L_2=\ol_{\pr^1}(1),L_3=\ol_{\pr^1}$ yields the same Fano $3$-fold $X$ as 
$L_1=L_2=\ol_{\pr^1}(-1),L_3=\ol_{\pr^1}$ and hence as $L_1=\ol_{\pr^1}(1),L_2=L_3=\ol_{\pr^1}$.

In conclusion, case $A$ yields two distinct Fano $3$-folds. The trivial choice
$L_1=L_2=L_3=\ol_{\pr^1}$ yields $Z=\pr^2\times \pr^1$ and $X= (\Bl_3\pr^2)\times\pr^1$.

The choice $L_1=\ol_{\pr^1}(1),L_2=L_3=\ol_{\pr^1}$ yields $Z=\pr_{\pr^1}(\ol(1)\oplus\ol\oplus\ol)$, the blow-up of $\pr^3$ along a line, and $X$ is the blow-up of $Z$ along three curves $S_1,S_2,S_3$, where
$S_1$ is the transform of a general line in $\pr^3$, while $S_2$ and $S_3$ are non-trivial fibers of the blow-up $Z\to\pr^3$.
 
 In case $B$ we have $a_2=a_3=0$, $|a_1|\leq 1$ by condition \eqref{condition}, $a_1\neq 0$ because $L_1\not\cong\ol_T$ by assumption, and finally $a_1>0$ because $2L_1$ is effective by Lemma \ref{effective}; thus the only possibility is $a_1=1$, and again $Z=\pr_{\pr^1}(\ol(1)\oplus\ol\oplus\ol)$. We note that $L_1=\ol_{\pr^1}(1)$ is globally generated, so that the additional assumption for case $B$ is fulfilled.
Then
$X$ is the blow-up of $Z$ along two curves $S_2,S_3$, which as before are non-trivial fibers of the blow-up $Z\to\pr^3$, and  the transform of a general conic in $\pr^3$.
\end{example}  
We have recovered the classification of Fano $3$-folds with $\rho_X=5$ \cite[Table 12.6]{fanoEMS}, as follows.
\begin{corollary}\label{class}
  There are three Fano $3$-folds with $\rho_X=5$, they have $\delta_X=3$ and are the following:
\begin{enumerate}[1)]
\item
  $X=(\Bl_3\pr^2)\times\pr^1$.

  \noindent This is \cite[Table 12.6, No.\ 3]{fanoEMS}, and it is obtained with Construction \ref{construction}, $T=\pr^1$, case $A$,
    $L_1=L_2=L_3=\ol$.
\item
  $X$ is the blow-up of $\pr^3$ along two skew lines and along two non-trivial fibers of the blow-up of one of the lines.

  \noindent This is \cite[Table 12.6, No.\ 2]{fanoEMS}, and is obtained with Construction \ref{construction},  $T=\pr^1$, case $A$,
$L_1=\ol(1)$, $L_2=L_3=\ol$.
\item
  $X$ is the blow-up of $\pr^3$ along a line and a conic (disjoint from each other), and along two non-trivial fibers of the blow-up of the line.

  \noindent This is \cite[Table 12.6, No.\ 1]{fanoEMS}, and is obtained with Construction \ref{construction},  $T=\pr^1$,
  case $B$, $L_1=\ol(1)$.
  \end{enumerate}
\end{corollary}  
In a similar way, Theorem \ref{delta3} yields the classification of Fano $4$-folds with $\delta_X=3$.
Now $T$ is a del Pezzo surface with $\delta_T\leq 3$; on the other hand $\rho_T=\delta_T+1$ (see Remark \ref{surfaces}), so that $1\leq\rho_T\leq 4$ and  $5\leq\rho_X\leq 8$. Here we sum up the outcome.
\begin{proposition}[\cite{delta3_4folds} and \cite{delta3}, Prop.~1.5]
  \label{classification}
  Let $X$ be a Fano $4$-fold with $\delta_X=3$. Then $5\leq \rho_X\leq 8$ and there are 19  families for $X$, among which 14 are toric. More precisely:
  \begin{enumerate}[--]
  \item if $\rho_X=8$, then $X\cong \Bl_3\pr^2\times \Bl_3\pr^2$;
  \item if     $\rho_X=7$, then $X\cong \Bl_3\pr^2\times \Bl_2\pr^2$;
  \item  if     $\rho_X=6$, there are 11  families for $X$, among which 8 are toric, and one is $\pr^1\times Y$, $Y$ the non-toric Fano $3$-fold with $\rho_Y=5$;
     \item  if     $\rho_X=5$, there are 6  families for $X$, among which 4 are toric.
    \end{enumerate}
  \end{proposition}
More generally,
given
the classification of $(n-2)$-dimensional Fano varieties, Theorem \ref{delta3} allows -- in principle -- to completely classify $n$-dimensional Fano varieties with Lefschetz defect $3$.
\section{Fano varieties with Lefschetz defect $2$}\label{sec2}
\noindent
After Theorems \ref{codim} and \ref{delta3}, the structure of Fano varieties with $\delta_X\geq 3$ is well understood. We consider in this section the next case, $\delta_X=2$, studied in \cite{codimtwo}. The following result shows that $X$ still has some interesting properties.
\begin{thm}[\cite{codimtwo}, Th.~1.2]\label{delta2}
Let $X$ be a smooth Fano variety with $\delta_X = 2$. Then one of the following
holds.
\begin{enumerate}[$(i)$]
  \item
    There exist a small modification $X\dasharrow X'$ and a conic bundle
    $X' \to Y$ where $X'$ and
$Y$ are smooth, and $\rho_X - \rho_Y = 2$.
\item There is an equidimensional fibration in del Pezzo surfaces $\psi\colon X \to T$, where $T$ has locally factorial, canonical singularities, $\codim \Sing(T) \geq 3$,
  and  $\rho_X-\rho_T=3$.
\end{enumerate}
\end{thm}
It is easy to find examples (among toric Fano varieties) of $X$ as in Theorem \ref{delta2}$(i)$ where a
birational modification is necessary, namely $X$ itself does not have a conic bundle structure,
nor a fibration in del Pezzo surfaces; for instance the two toric Fano $4$-folds of combinatorial type $J$ in Batyrev's classification of toric Fano $4$-folds \cite[3.2.7]{bat2}.

We do not know whether the above result is optimal; indeed, based on examples, it may very well be that $(i)$ holds for every Fano variety with $\delta_X=2$. The author believes that the case $\delta_X=2$ has yet to be fully studied and understood.
\begin{example}[Fano $3$-folds with $\delta_X=2$]
  Let $X$ be a Fano $3$-fold with $\delta_X=2$; we have $\rho_X\in\{3,4\}$, see Table \ref{table3folds} on page \pageref{table3folds}.

  If $\rho_X=3$, then $\delta_X=\rho_X-1$ is maximal, and the structure theorem  \cite[Theorem 3.8]{minimal} yields all the possibilities for $X$; in particular there exists a conic bundle $X\to\pr^2$. These are 6 among the 31 families of Fano $3$-folds with $\rho_X=3$ (see \cite[Table 12.4]{fanoEMS}), the remaining ones have $\delta_X=1$.

Again by Table \ref{table3folds}, all Fano $3$-folds with $\rho_X=4$ have $\delta_X=2$; there are 13 families \cite[Table 12.5]{fanoEMS}, \cite{morimukaierratum}. By \cite[Theorem on p.\ 141, Proposition 7.1.11]{fanoEMS}, there is a conic bundle $X\to Y$ where $Y$ is a smooth del Pezzo surface. It is not difficult to see that $Y\not\cong\pr^2$, because the discriminant curve should have at least two connected components, so that $\rho_Y\in\{2,3\}$. If $\rho_Y=3$, then $Y\cong \text{Bl}_2\pr^2$ and
$X\cong\pr^1\times \Bl_2\pr^2$ by \cite[Theorem 7.1.15]{fanoEMS}, so that we can replace $X\to Y$ with the conic bundle $X\to \pr^1\times \mathbb{F}_1\to \pr^1\times\pr^1$. In any case we get a conic bundle $X\to Y$ with $\rho_{Y}=2$.

We summarize this example in the following table.

\begin{table}[!h]\caption{Fano $3$-folds with $\delta=2$}\label{table5}
$$\begin{array}{||c|c|c||}
\hline\hline
   
    \rho=3   & \text{6 families} & \exists\text{ conic bundle } X\to\pr^2\\
    
    \hline

   \rho=4   & \text{13 families} & \exists\text{ conic bundle }X\to\pr^1\times\pr^1\text{ or }X\to\mathbb{F}_1\\
   
\hline\hline
  \end{array}$$
\end{table}

  We conclude that all Fano $3$-folds with $\delta_X=2$ have a conic bundle $X\to Y$ with $\rho_X-\rho_Y=2$, hence satisfy $(i)$ of Theorem \ref{delta2}.
\end{example}
In the $4$-dimensional case, we have a more refined version of Theorem \ref{delta2}.
\begin{thm}\label{delta2dim4}
  Let $X$ be a Fano $4$-fold with $\delta_X=2$. Then $\rho_X\leq 12$.

  If moreover $\rho_X\geq 7$, then there exist a small modification $X\dasharrow X'$ and a conic bundle
    $f \colon X' \to Y$ where $X'$ and
$Y$ are smooth,  $\rho_X - \rho_Y = 2$, and $Y$ is \emph{weak Fano}, i.e.\ $-K_Y$ is nef and big.
\end{thm}
The bound $\rho_X\leq 12$ is obtained by studying the geometry of the weak Fano $3$-fold $Y$ and proving that $\rho_Y\leq 10$, see
 \cite[Theorem 4.4]{eff}.
\begin{proof}
By  \cite[Theorem 1.2]{cdue} we have $\rho_X\leq 12$, and if $\rho_X\geq 7$, there exists a diagram:
 $$\xymatrix{X\ar@{-->}[r]\ar[d]_g & {X'}\ar[d]\ar[dr]^f &\\
   {X_1}\ar@{-->}^h[r] & {X_1'}\ar[r]^{f_1} &Y
 }$$
 where all varieties are smooth, $X_1$ is Fano, $g$ is the blow-up of a smooth surface, $h$ is a small modification, and $f_1$ is an elementary conic bundle.
 
  We show that $Y$ is weak Fano.
 It follows from \cite[Lemma 2.8]{prokshok} that $Y$ is of Fano type, in particular
 $-K_Y$ is big.

 Let $R$ be an extremal ray of $\NE(Y)$, and $\ph\colon Y\to Y'$ the associated elementary contraction. If $\ph$ is of fiber type, then $-K_Y\cdot R>0$, because of bigness. If $\ph$ 
 is birational and small, then $-K_Y\cdot R=0$ by \cite[Lemma 4.5]{eff}.

 Finally suppose that $\ph$ is birational divisorial. We show that the composition $f_1\circ h\colon X_1\dasharrow Y$ cannot be regular, so that $-K_Y\cdot R>0$ by \cite[Lemma 4.6]{eff}. Indeed 
 let $E\subset Y$ be the exceptional divisor of $\ph$. Since $\dim \ph(E)\leq 1$, as in Example \ref{es3fold} we see that  $\dim\N(E,Y)\leq 2$. If $f_1\circ h$ were regular, let $E_X\subset X$ be a prime divisor whose image in $Y$ is $E$. Then $\dim\N(E_X,X)\leq \rho_X-\rho_Y +\dim\N(E,Y)\leq 4$, which contradicts the assumptions $\delta_X=2$ and $\rho_X\geq 7$.

 In conclusion $-K_Y\cdot R\geq 0$ for every extremal ray $R$ of $\NE(Y)$, hence $-K_Y$ is nef.
\end{proof}
We can deduce from the previous results some general properties of Fano varieties with $\delta_X\geq 2$.
\begin{lemma}\label{properties}
  Let $X$ be a smooth Fano variety with $\delta_X\geq 2$. Then the following hold:
  \begin{enumerate}[$(a)$]
    \item
  $X$ is covered by 
  a family of rational curves having intersection $2$  with $-K_X$;
\item  $X$ contains a smooth prime divisor $E$, with a $\pr^1$-bundle structure, such that $E\cdot f=-1$, where $f\subset E$ is a fiber of the $\pr^1$-bundle.
\end{enumerate}

If instead $X$ has $\delta_X=1$, then at least one of $(a)$ or $(b)$ holds.
\end{lemma}
\begin{proof}
  Statement $(b)$ follows from \cite[Proposition 2.5]{codim}.
 When $\delta_X\geq 4$ (respectively $\delta_X=3$), $(a)$ is a straightforward conseguence of Theorem \ref{codim} (respectively Theorem \ref{delta3}).

We show $(a)$ when $\delta_X=2$. We apply 
  \cite[Theorem 5.2]{codimtwo}, which is a more refined version of 
  Theorem \ref{delta2}. By this result, 
in case $(i)$ of Th.~\ref{delta2}, the general fiber of the conic bundle $X'\to Y$ is contained in the open subset where $X\dasharrow X'$ is an isomorphism, so it yields the desidered family in $X$. Moreover, in case $(ii)$ of Th.~\ref{delta2}, the fibration $\psi\colon X\to T$ is ``quasi-elementary'', which means that for a general fiber $F$ we have $\dim\N(F,X)=\rho_X-\rho_T$. Then $F$ is a del Pezzo surface with $\rho_F\geq  \rho_X-\rho_T=3$, in particular it has a covering
  family of rational curves of anticanonical degree $2$, and again this yields the desired family in $X$.

  Finally the statement for the case $\delta_X=1$ follows from \cite[Proposition 6.1]{codimtwo}.
\end{proof}
\section{An overview of known Fano $4$-folds with $\rho\geq 6$}\label{sec4folds}
\noindent
In this section we describe the known examples of Fano $4$-folds with $\rho\geq 6$.
In fact, they are very few:
we have products, toric examples, two additional families with $\delta=3$ (see Proposition \ref{classification}), one $4$-fold contained in a product of grassmannians \cite{manivel4fold}, and four families constructed from the blow-ups of $\pr^4$ in points. Let us start with these last ones.
\begin{example}[Fano models of the blow-ups of $\pr^4$ in up to 8 points]\label{points}
  Let $\Bl_m\pr^4$ be the blow-up of $\pr^4$ in $m\geq 2$ general points. This $4$-fold  is not Fano, as the transform $\ell$ of a line through 2 blown-up points has $-K\cdot\ell=-1$. On the other hand, for $m\leq 8$ there exists a
  small modification $\Bl_m\pr^4\dasharrow X_m$ such that $X_m$ is smooth and Fano.

  Let us describe  this birational map.  Consider in $\Bl_m\pr^4$ the transforms of the $\binom{m}{2}$ lines through 2 blown-up points, and, for $m\in\{7,8\}$,
of  the $\binom{m}{7}$ rational normal quartics through 7 blown-up points. These curves in $\Bl_m\pr^4$ are pairwise disjoint, and are smooth rational curves with normal bundle $\ol_{\pr^1}(-1)^{\oplus 3}$.

  Let $Y_m\to \Bl_m\pr^4$ be the blow-up of all these curves. The exceptional divisors are isomorphic to $\pr^1\times\pr^2$ with normal bundle $\ol(-1,-1)$. Moreover $Y_m$ is also obtained  by blowing-up a smooth $4$-fold $X_m$ along some pairwise disjoint  smooth rational surfaces, with the same exceptional divisors.

  The $4$-fold $X_m$ is Fano, has $\rho_{X_m}=m+1\leq 9$, and it is toric if and only if $m\leq 5$.
  For $m=7$, $X_7$ is isomorphic to one of K\"uchle's Fano $4$-folds \cite{kuchle}, more precisely (b9), see \cite{parabolic}.
  For $m=8$, this example has been constructed in \cite{vb}, and is studied in detail in \cite{zhixin}.
  
  It follows from the description of the generators of the cone of effective divisors of $X_8$ \cite[\S 6.5]{vb} that $X_m$ cannot contain a prime divisor $E$ with a $\pr^1$-bundle structure such that $E\cdot f=-1$, where $f\subset E$ is a fiber of the $\pr^1$-bundle. Therefore $\delta_{X_m}\leq 1$ by Lemma \ref{properties}.
  
  On the other hand, for $m\leq 5$, $X_m$ is toric and it is easy to see that it contains pairs of disjoint prime divisors, so that  $\delta_{X_m}= 1$ in this case, by Example \ref{disjoint}.

  Note that $\Bl_m\pr^4$ is not a Mori dream space for $m\geq 9$  \cite{mukaiXIV}, therefore there cannot
be a small modification to
 a smooth Fano $4$-fold, and this construction works only for $m\leq 8$.
\end{example}
Let us consider now the toric case. Toric Fano $4$-folds are classified \cite{bat2,sato}, they are 124, and have $\rho\leq 8$; we list the ones with $6\leq\rho\leq 8$ in
 Table \ref{table}, where $X_5$ is from Example \ref{points}, while the combinatorial type $U$ refers to Batyrev's notation in \cite{bat2}.
\begin{table}[!h]\caption{Toric Fano $4$-folds with $\rho\geq 6$}\label{table}
$$\begin{array}{||c|c|c||}
\hline\hline
    
    \rho=8   & \Bl_3\pr^2\times \Bl_3\pr^2 & \delta=3 \\
    
\hline

    \rho=7 &  \Bl_3\pr^2\times \Bl_2\pr^2 & \delta=3 \\

    \hline

\rho=6 &   \Bl_2\pr^2\times  \Bl_2\pr^2 & \delta=2 \\

\cline{2-3}

     & \text{combinatorial type $U$ (8 varieties)} &  \delta=3\\

\cline{2-3}

     & X_5   &  \delta=1\\
    
\hline\hline
  \end{array}$$
\end{table}

In the recent paper \cite{manivel4fold}, Manivel constructs a Fano $4$-fold $X_L$ with $\rho_{X_L}=6$, contained in the product $\text{Gr}(1,\pr^3)\times\text{Gr}(2,\pr^4)$. Its numerical invariants are given in \cite[Prop.~4.1]{manivel4fold}, and  using these one can check that this $4$-fold is different from the other ones mentioned in this section.

Finally, 
excluding products and toric ones, to our knowledge the  known examples of Fano $4$-folds with $\rho\geq 6$ are the $4$-folds $X_6,X_7,X_8$ of Example \ref{points}, the $4$-fold $X_L$, and the two new Fano $4$-folds $X_{B_1}$ and $X_{B_2}$ with $\delta_X=3$ and $\rho_X=6$ constructed in  \cite[\S 7]{delta3} (see Proposition \ref{classification}), as shown in Table \ref{table2}.
\begin{table}[!h]\caption{Known Fano $4$-folds with $\rho\geq 6$, excluding toric ones and products}\label{table2}
$$\begin{array}{||c|c|c||}
\hline\hline
   
    \rho=7,8,9   & X_{\rho-1} & \delta\in\{0,1\} \\
    
\hline

    \rho=6 & X_{B_1} & \delta=3 \\

\cline{2-3}

     & X_{B_2} &  \delta=3\\

    \cline{2-3}

                & X_L & \delta\in\{0,1,2\}\\
    
\hline\hline
  \end{array}$$
\end{table}

In the references \cite{bat2,delta3_4folds, delta3,manivel4fold} one can find the main numerical invariants of the $4$-folds in Tables \ref{table} and \ref{table2}, such as Hodge numbers, $K_X^4$, and $h^0(X,-K_X)$. Note that all these varieties are rational.

\medskip

In particular we see that all known Fano $4$-folds with $\delta_X=2$ have $\rho_X\leq 6$, quite far from the bound $\rho_X\leq 12$ of Theorem \ref{delta2dim4}. It would be very interesting to construct new examples, maybe starting from the geometric description of the case $\rho_X\geq 7$ given by
Theorem \ref{delta2dim4}.

Finally, it is still an open question to determine a good bound on $\rho_X$ when $\delta_X\in\{0,1\}$, see \cite{small} for related results. It is expected that $\rho_X\leq 18$ for every Fano $4$-fold, and that  $X$ is a product of surfaces when $\rho_X$ is close to $18$.
\section{Special Fano's: index $>1$ and the toric case}\label{secindex}
\noindent We recall that the \emph{index} $r_X$ of a Fano variety is the divisibility of $-K_X$ in $\Pic(X)$, namely
$$r_X=\max\{m\in\Z_{>0}\,|\,-K_X\sim mH\text{ for }H\in\Pic(X)\}.$$
  A related invariant is
the \emph{pseudo-index}, defined
 as   $$\iota_X=\min\{-K_X\cdot C\,|\,C\text{ a rational curve in }X\},$$
which is a multiple of $r_X$.

One expects ``most'' Fano varieties to have $r_X=\iota_X=1$, and that Fano varieties with $\iota_X>1$ or $r_X>1$ should be simpler. For instance, Fano 4-folds with $r_X>1$ are classified and have $\rho_X\leq 4$, see \cite[Cor.~3.1.15, Table 12.1, \S 5.2, and Table 12.7]{fanoEMS}, so that the missing step in the classification of Fano $4$-folds is the case of index $r_X=1$.

With respect to the Lefschetz defect, we have the following.
    \begin{thm}[\cite{codim}, Th.~1.2]
      Let $X$ be a smooth Fano variety with  $r_X>1$, or more generally with $\iota_X>1$. Then one of the following holds:
      \begin{enumerate}[$(i)$]
      \item $\delta_X=0$;
      \item $\delta_X=1$, $\iota_X=2$, $r_X\in\{1,2\}$, and there exists a smooth morphism $X\to Y$ with fiber $\pr^1$, where $Y$ is smooth, Fano, and $\iota_Y>1$.
        \end{enumerate}
      \end{thm}
 Let us consider now the case of toric Fano varieties; we have the following.
\begin{proposition}\label{toric}
  Every smooth toric Fano variety $X$ has $\delta_X\leq 3$.

  Moreover, if $\delta_X=3$, then there exist a toric Fano variety $T$  and $L_1,L_2,L_3\in\Pic(T)$ with
$-K_T+L_i-L_j$ ample for every $i,j=1,2,3$, such that  $X$ is obtained by blowing-up $Z:=\pr_T(L_1\oplus L_2\oplus L_3)$ along the three torus invariant sections of the $\pr^2$-bundle $Z\to T$.
\end{proposition}  
\begin{proof}
  If $X$ is a Fano variety with $\delta_X\geq 4$, then by Theorem \ref{codim} we have $X\cong S\times T$ where $S$ is a del Pezzo surface with $\rho_S=\delta_X+1\geq 5$. In particular $S$ is not toric, as toric del Pezzo surfaces have $\rho\leq 4$, therefore $X$ is not toric either.

 Now if $X$ is a toric Fano variety with $\delta_X=3$, then we can apply Theorem \ref{delta3}; moreover, since $X$ is toric, only case $A$ can occur, and $T$ must be toric too, see \cite[Remark 6.1]{delta3}.
\end{proof}  
It is interesting to note that Proposition \ref{toric}
is very much related to
\cite[Theorem 3.4]{fano},  proved much earlier with completely combinatorial methods. This last result says that for every torus invariant prime divisor $D\subset X$, one has $\codim\N(D,X)\leq 3$, and if there is a $D$ with 
$\codim\N(D,X)= 3$, then $X$ is a toric $(\Bl_3\pr^2)$-bundle onto a toric Fano variety $T$.
If instead $\codim\N(D,X)= 2$, then one gets a statement similar
to Theorem \ref{delta2}.
\section{Beyond smooth Fano's}\label{secgen}
\noindent The notion of Lefschetz defect can be introduced for any  projective variety; anyway, the good properties of $\delta_X$ strongly depend on the Fano assumption. Indeed,  it seems that already when $X$ is weak Fano  there is no analog of Theorem \ref{codim}, as the following example shows.
\begin{example}[a toric weak Fano $3$-fold]
  After \cite{kreuzerskarke}
  we know that $\rho_X\leq 35$
for every  smooth, toric weak Fano $3$-fold $X$, and the bound is sharp.
One example with $\rho_X=35$ is described in detail in 
\cite[\S 7]{anna}, together with the extremal rays of $\NE(X)$. In particular, this $X$ contains a  divisor isomorphic to $\pr^2$, so that $\delta_X=34$.
\end{example}  
For singular Fano varieties, we have the following.
\begin{thm}[\cite{gloria}]
  Let $X$ be a $\Q$-factorial Gorenstein Fano variety, with canonical singularities, and at most finitely many non-terminal points. Then $\delta_X\leq 8$.

  Moreover, if $\delta_X\geq 4$, then there is a finite morphism $X\to S\times T$, where $S$ and $T$ are normal with rational singularities, and $\dim S=2$.
\end{thm}
The main point where the technique used in the proof of Theorem \ref{codim} breaks, when $X$ is too singular or not Fano, is the following. As explained in \S \ref{MMP}, the properties of elementary birational contraction with fibers of dimension $\leq 1$ play a crucial role in the construction which is the starting point for the proof of the theorem. In particular, if $X$ is smooth and Fano, such a contraction is always divisorial. Instead,
when $X$ is  weak Fano, it can have small (crepant) elementary contractions having fibers of dimension $\leq 1$. Similarly, if $X$ is a $\Q$-factorial Gorenstein Fano variety with canonical singularities, again it can have small  elementary contractions having fibers of dimension $\leq 1$, see \cite[Example 2.11]{gloria}.


\begin{thebibliography}{10}

\bibitem{ando}
T.~Ando, \emph{On extremal rays of the higher dimensional varieties}, Invent.\
  Math.\ \textbf{81} (1985), 347--357.

\bibitem{bat2}
V.V. Batyrev, \emph{On the classification of toric {F}ano 4-folds}, J.\ Math.\
  Sci.\ (New York) \textbf{94} (1999), 1021--1050.

\bibitem{BCHM}
C.~Birkar, P.~Cascini, C.D. Hacon, and J.~McKernan, \emph{Existence of minimal
  models for varieties of log general type}, J.\ Amer.\ Math.\ Soc.\
  \textbf{23} (2010), 405--468.

\bibitem{bonwisncamp}
L.~Bonavero, F.~Campana, and J.A. Wi{\'s}niewski, \emph{Vari{\'e}t{\'e}s
  complexes dont l'{\'e}clat{\'e}e en un point est de {F}ano}, C.~R., Math.,
  Acad.~Sci.~Paris \textbf{334} (2002), 463--468.

\bibitem{fano}
C.~Casagrande, \emph{Toric {F}ano varieties and birational morphisms}, Int.\
  Math.\ Res.\ Not.\ \textbf{27} (2003), 1473--1505.

\bibitem{31}
\bysame, \emph{On {F}ano manifolds with a birational contraction sending a
  divisor to a curve}, Michigan Math.\ J.\ \textbf{58} (2009), 783--805.

\bibitem{codim}
\bysame, \emph{On the {P}icard number of divisors in {F}ano manifolds},
  Ann.~Sci.~{\'E}c.~Norm.~Sup{\'e}r.\ \textbf{45} (2012), 363--403.

\bibitem{eff}
\bysame, \emph{\noop{aaa}{O}n the birational geometry of {F}ano 4-folds},
  Math.~Ann.\ \textbf{355} (2013), 585--628.

\bibitem{cdue}
\bysame, \emph{\noop{zzz}{N}umerical invariants of {F}ano 4-folds},
  Math.~Nachr.\ \textbf{286} (2013), 1107--1113.

\bibitem{codimtwo}
\bysame, \emph{On some {F}ano manifolds admitting a rational fibration}, J.\
  London Math.\ Soc.\ \textbf{90} (2014), 1--28.

\bibitem{parabolic}
\bysame, \emph{{R}ank {$2$} quasiparabolic vector bundles on {$\pr^1$} and the
  variety of linear subspaces contained in two odd-dimensional quadrics},
  Math.\ Z.\ \textbf{280} (2015), 981--988.

\bibitem{small}
\bysame, \emph{Fano 4-folds with a small contraction}, Adv.\ Math.\
  \textbf{405} (2022), 1--55, paper no.\ 108492.

\bibitem{vb}
C.~Casagrande, G.~Codogni, and A.~Fanelli, \emph{The blow-up of {$\pr^4$} at
  {$8$} points and its {F}ano model, via vector bundles on a del {P}ezzo
  surface}, Rev.\ M{\'a}t.\ Complut.\ \textbf{32} (2019), 475--529.

\bibitem{minimal}
C.~Casagrande and S.~Druel, \emph{Locally unsplit families of large
  anticanonical degree on {F}ano manifolds}, Int.\ Math.\ Res.\ Not.\
  \textbf{2015} (2015), 10756--10800.

\bibitem{delta3_4folds}
C.~Casagrande and E.A. Romano, \emph{Classification of {F}ano 4-folds with
  {L}efschetz defect 3 and {P}icard number 5}, J.\ Pure Appl.\ Algebra
  \textbf{226} (2022), no.~3, 13 pp., paper no.\ 106864.

\bibitem{delta3}
C.~Casagrande, E.A. Romano, and S.A.Secci, \emph{Fano manifolds with
  {L}efschetz defect 3}, J.\ Math.\ Pures Appl.\ \textbf{163} (2022), 625--653,
  corrigendum: published online 13 October 2022.

\bibitem{gloria}
G.~{Della Noce}, \emph{\noop{aaa}{O}n the {P}icard number of singular {F}ano
  varieties}, Int.\ Math.\ Res.\ Not.\ \textbf{2014} (2014), 955--990.

\bibitem{hukeel}
Y.~Hu and S.~Keel, \emph{Mori dream spaces and {GIT}}, Michigan Math.\ J.\
  \textbf{48} (2000), 331--348.

\bibitem{fanoEMS}
V.A. Iskovskikh and Yu.G. Prokhorov, \emph{Algebraic geometry {V} - {F}ano
  varieties}, Encyclopaedia Math.\ Sci.\, vol.~47, Springer-Verlag, 1999.

\bibitem{kollarmori}
J.~Koll{\'a}r and S.~Mori, \emph{Birational geometry of algebraic varieties},
  Cambridge Tracts in Mathematics, vol. 134, Cambridge University Press, 1998.

\bibitem{kreuzerskarke}
M.~Kreuzer and H.~Skarke, \emph{Classification of reflexive polyhedra in three
  dimensions}, Adv.\ Theor.\ Math.\ Phys.1 \textbf{2} (1998), 853--871.

\bibitem{kuchle}
O.~K{\"u}chle, \emph{On {F}ano {$4$}-folds of index {$1$} and homogeneous
  vectos bundles over {G}rassmannians}, Math.\ Z.\ \textbf{218} (1995),
  563--575.

\bibitem{manivel4fold}
L.~Manivel, \emph{A four-dimensional cousin of the {S}egre cubic}, preprint
  arXiv:2211.16154, 2022.

\bibitem{morimukai2}
S.~Mori and S.~Mukai, \emph{Classification of {F}ano {$3$}-folds with $b_2\geq
  2$, {I}}, Algebraic and Topological Theories -- to the memory of Dr.\
  Takehiko Miyata (Kinosaki, 1984), Kinokuniya, Tokyo, 1986, pp.~496--545.

\bibitem{morimukaierratum}
\bysame, \emph{Erratum: ``{C}lassification of {F}ano {$3$}-folds with $b_2\geq
  2$''}, Manuscr.~Math.\ \textbf{110} (2003), 407.

\bibitem{mukaiXIV}
S.~Mukai, \emph{Counterexample to {H}ilbert's fourteenth problem for the
  $3$-dimensional additive group}, RIMS Preprint n.~1343, Kyoto, 2001.

\bibitem{prokshok}
Yu.G. Prokhorov and V.V. Shokurov, \emph{Towards the second main theorem on
  complements}, J.\ Algebraic Geom.\ \textbf{18} (2009), 151--199.

\bibitem{sato}
H.~Sato, \emph{Toward the classification of higher-dimensional toric {F}ano
  varieties}, T{\^o}hoku Math.\ J.\ \textbf{52} (2000), 383--413.

\bibitem{anna}
A.~Scaramuzza, \emph{Algorithms for projectivity and extremal classes of a
  smooth toric variety}, Experiment.\ Math.\ \textbf{18} (2009), 71--84.

\bibitem{saverio}
S.A. Secci, \emph{Fano 4-folds having a prime divisor of {P}icard number 1},
  preprint arXiv:2103.16140. To appear in Advances in Geometry, 2022.

\bibitem{toru}
T.~Tsukioka, \emph{Classification of {F}ano manifolds containing a negative
  divisor isomorphic to projective space}, Geom.~Dedicata \textbf{123} (2006),
  179--186.

\bibitem{wisn}
J.A. Wi{\'s}niewski, \emph{On contractions of extremal rays of {F}ano
  manifolds}, J.~Reine Angew.~Math.\ \textbf{417} (1991), 141--157.

\bibitem{zhixin}
Z.~Xie, \emph{Anticanonical geometry of the blow-up of {$\pr^4$} in 8 points
  and its {F}ano model}, Math.\ Z., published online 19 September, 2022.

\end{thebibliography}
\providecommand{\noop}[1]{}
\providecommand{\bysame}{\leavevmode\hbox to3em{\hrulefill}\thinspace}
\providecommand{\MR}{\relax\ifhmode\unskip\space\fi MR }
\providecommand{\MRhref}[2]{%
  \href{http://www.ams.org/mathscinet-getitem?mr=#1}{#2}
}
\providecommand{\href}[2]{#2}

\end{document}